\documentclass[10pt]{article}

\usepackage{amsmath}
\usepackage{amsfonts}
\usepackage{amssymb}
\usepackage{framed}
\usepackage{enumerate}
\usepackage{scrextend}
\usepackage{amsthm}
\usepackage{bbm}
\usepackage[margin=1.5in]{geometry}
\usepackage{booktabs}
\usepackage[numbers,sort&compress]{natbib}
\usepackage{url}

\usepackage{tikz-cd}

\newcommand{\mult}{\textrm{mult}}
\newcommand{\coeff}[2]{\left[ #1  \right]_{#2}}

\newcommand{\ep}{\varepsilon}
\newcommand{\Jac}{\textrm{Jac}}
\newcommand{\val}{\textrm{val}}

\newcommand\restr[2]{{
	\left.\kern-\nulldelimiterspace
	#1
	\vphantom{\big|}
	\right|_{#2}
	}}

\newsavebox{\overlongequation}
\newenvironment{centerlongequation}
 {\begin{displaymath}\begin{lrbox}{\overlongequation}$\displaystyle}
 {$\end{lrbox}\makebox[0pt]{\usebox{\overlongequation}}\end{displaymath}}

\theoremstyle{plain}
\newtheorem{thm}{Theorem}[section]
\newtheorem{lem}[thm]{Lemma}

\theoremstyle{definition}

\theoremstyle{remark}
\newtheorem{rem}[thm]{Remark}

\title{Explicit Rational Group Law on Hyperelliptic Jacobians of any Genus}
\author{David Urbanik}

\begin{document}
\maketitle

\begin{abstract}
It is well-known that abelian varieties are projective, and so that there exist explicit polynomial and rational functions which define both the variety and its group law. It is however difficult to find any explicit polynomial and rational functions describing these varieties or their group laws in dimensions greater than two. One exception can be found in Mumford's classic ``Lectures on Theta'', where he describes how to obtain an explicit model for hyperelliptic Jacobians as the union of several affine pieces described as the vanishing locus of explicit polynomial equations. In this article, we extend this work to give explicit equations for the group law on a dense open set. One can view these equations as generalizations of the usual chord-based group law on elliptic curves.
\end{abstract}

\section{Introduction}

Abelian varieties and their equations have long attracted interest in arithmetic geometry. Although it is known that equations describing these varieties must exist, and their nature has received some study\cite{oneqdef1, oneqdef2, oneqdef3}, it is in general believed to be impractical or infeasible to write such equations down. This attitude is perhaps best summarized in Milne's notes\cite{milneAV} on the subject, where he writes ``In general, it is not possible to write down explicit equations for an abelian variety of dimension $>$ 1, and if one could, they would be too complicated to be of use.'' 

A brief look at the literature on the matter seems to justify this outlook. For instance, the paper of Flynn\cite{flynngenus2eqs} gives a general set of equations for genus two Jacobians over an arbitrary ground field; there are $72$ equations in total, listed in an appendix, which describe these Jacobians as projective subvarieties of a $15$-dimensional projective space. A follow-up paper from Flynn\cite{flynngenus2gplaw} describes the group law, the equations of which he describes as ``too large to be written down,'' and instead focuses on methods to compute specializations of the group law for tasks such as point-doubling or the addition of fixed points of low order. Related work by Grant\cite{grant} gives a simpler set of defining equations in $8$-dimensional projective space, but at the cost of some generality.\footnote{Although, in fairness to Grant, our work makes similar assumptions.} In both cases, the authors remark that portions of their exposition required computer verification, as the algebraic expressions involved are too complicated to be reliably manipulated by hand.

One of the difficulties that arises in these approaches is that the usual methods for embedding genus $g$ Jacobians into $n$-dimensional projective spaces tend to result in an exponential dependence of $n$ on $g$, with $n = 3^{g} - 1$ and $n = 4^{g} - 1$ being common (as in the case for $g = 2$ above). This ensures that finding explicit equations via this strategy must necessarily be impractical for large $g$. An alternative approach, which we pursue in this paper, is to give explicit equations for Jacobians and their group law affine-locally, and construct the full Jacobian by gluing of charts. For hyperelliptic curves, the Jacobian variety itself is described in this manner by Mumford\cite{lecontheta2}, with the affine-local pieces utilizing affine spaces of dimension $3g + 1$, and hence with the number of parameters depending only linearly on $g$. In this paper, we show how to extend this construction to give explicit equations for the group law.

Our work is inspired by the paper of Leitenberger\cite{lieten} and the paper of Costello and Lauter\cite{costlaut}, both of which essentially carry out this approach in the case $g = 2$. Our methods can be viewed as a substantial generalization of their work.

\section{Algebraic Construction of Hyperelliptic Jacobians}
\label{jacSection}

In this section, we review the construction of hyperelliptic Jacobians that appears in Mumford's Lectures on Theta\cite{lecontheta2} and set notation. We consider hyperelliptic curves $C$ defined over an algebraically closed field $\mathbbm{k}$ with $\textrm{char} \hspace{0.1em} \mathbbm{k} \neq 2$ by two equations of the form
\begin{align*}
C_{1} : y^2 &= f(x) = f_{2g+1} x^{2g+1} + f_{2g} x^{2g} + \cdots + f_{0} \\
C_{2} : s^2 &= h(t) = t(h_{2g+1}t^{2g+1} + h_{2g} t^{2g} + \cdots + h_{0})
\end{align*}
glued along the morphism which makes the identifications $x = 1/t$ and $y = s/t^{g+1}$. We require that $f(x)$ has non-vanishing discriminant, and that $f$ is monic.\footnote{For the formulas we will develop, it will be useful to consider all the coefficients of $f$ on ``equal footing,'' which is why we give the $x^{2g+1}$ coefficient a distinct label despite the fact that we will always assume it is equal to $1$.} Note that the equation $C_{2}$ completes the curve defined by $C_{1}$ by adding a single ``infinite'' point corresponding to $(t, s) = (0, 0)$. Note also that $h_{i} = f_{(2g+1)-i}$. We will work with the equation $C_{1}$, and refer to the point $(t, s) = (0, 0)$ by the symbol $\infty$. We define the \emph{hyperelliptic involution} to be the map $\iota : C \to C$ determined by $(x, y) \mapsto (x, -y)$. If $P$ is a point on $C$, then $\iota(P)$ is deemed its \emph{conjugate}.

One may check that the curve $C$ is smooth, and that all divisor classes in $\Jac(C) = \textrm{Pic}^{0}(C)$ have a unique representative of the form $P_{1} + \cdots + P_{g} - g \infty$, where each $P_{i}$ is a point on $C$. To introduce coordinates into $\Jac(C)$, Mumford describes how to parametrize unordered $g$-tuples of points on $C_{1}$. Given $P_{i} = (x_{i}, y_{i})$, where $1 \leq i \leq g$ and $P_{i} \neq \iota(P_{j})$ for $i \neq j$, we define two polynomials describing the divisor $P_{1} + \cdots + P_{g}$. The first is defined as 

\[
u(x) = \prod_{i = 1}^{n} (x - x_{i}) = u_{g} x^{g} + u_{g-1} x^{g-1} + \cdots + u_{0} ;
\]
that is, it is the monic polynomial whose roots are the $x$-coordinates of the $P_{i}$'s counted with their multiplicity $\mult(P_{i})$. The second polynomial $v(x) = \sum_{i = 0}^{g-1} v_{i} x^{i}$ is defined to be the unique polynomial of degree $g-1$ which approximates the function $y$ up to order $\mult(P_{i})$ at $P_{i}$; that is, where $\val_{P_{i}}(v - y) = \mult(P_{i})$ for all $1 \leq i \leq g$, and $\val_{P_{i}}$ is the valuation at $P_{i}$. Note that for each $i$ where $y_{i} \neq 0$, the coordinate function $z_{i} = x - x_{i}$ is a uniformizer at $P_{i}$, and re-expressing the polynomial $v$ in terms of $z_{i}$ the condition $\val_{P_{i}}(v - y) = \mult(P_{i})$ amounts to imposing $\mult(P_{i})$ linear relations on the coefficients $v_{i}$. This gives $g$ linear relations total, which may be solved to find the coefficients of $v$. To see the uniqueness claim, observe that if $v_{1}$ and $v_{2}$ are any two such polynomials their difference $v_{1} - v_{2}$ satisfies
\[ \val_{P_{i}}(v_{1} - v_{2}) \geq \min \{ \val_{P_{i}}(v_{1} - y), \val_{P_{i}}(y - v_{2}) \} = \mult(P_{i}) . \]
But then $v_{1} - v_{2}$ is a polynomial of degree at most $g-1$ and has $g$ roots with multiplicity, hence must be zero.

The pairs $(u, v)$ are in one-to-one correspondence with degree $g$ effective divisors on $C_{1}$ not containing any pair of conjugate points: the roots of $u$ give the $x$-coordinates $x_{i}$ of the $g$ points, the value $v(x_{i})$ gives their $y$-coordinates $y_{i}$, and as we have seen the pair $(u, v)$ is uniquely determined. Moreover, we have
\[ \val_{P_{i}}(f - v^2) = \val_{P_{i}}(y^2 - v^2) = \mult(P_{i}) + \val_{P_{i}}(y + v) \geq \mult(P_{i}) , \]
where equality holds provided that $y_{i} \neq 0$ since then $y + v$ is non-vanishing at $P_{i}$. Hence $f - v^2$ is a polynomial in $x$ of degree $2g+1$ which vanishes to order $\mult(P_{i})$ at each point $P_{i}$, and so we have that $u | (f - v^2)$. Writing $w = \sum_{i = 0}^{g+1} w_{i} x^{i}$ for the unique monic degree $g+1$ polynomial which satisfies $f - v^2 = u w$, we get the following relations by examining the $x^{i}$ coefficient:
\begin{align}
\label{jaceqs}
f_{i} - \sum_{j = 0}^{i} v_{j} v_{i-j} = \sum_{j = 0}^{i} u_{j} w_{i-j} \hspace{6em} 0 \leq i \leq 2g+1 .
\end{align}
Here we have adopted a convention which will be in use throughout the paper, which is that polynomials may be regarded as formal power series in which all but finitely many coefficients, all of which have non-negative index, are zero. Thus we have that each of the sets of coefficients $f_{i}, u_{i}, v_{i}$ and $w_{i}$ are defined for all $i \in \mathbb{Z}_{\geq 0}$ (or, when it will be convenient, all $i \in \mathbb{Z}$), and so the equation ($\ref{jaceqs}$) holds for all $i \in \mathbb{Z}_{\geq 0}$, although it is only non-trivial when $0 \leq i \leq 2g+1$. 


The polynomials $u$, $v$ and $w$ have $g + g + (g+1) = 3g + 1$ undetermined coefficients among them, and as $i$ ranges from $0$ to $2g$ we obtain $2g+1$ relations from ($\ref{jaceqs}$), where we note that the relation obtained in the case $i = 2g+1$ is redundant. We have the following result from Mumford\cite{lecontheta2}:
\begin{thm}[Mumford]
\label{mumfDenseOpen}
The equations (\ref{jaceqs}) for $0 \leq i \leq 2g$ define a $g$-dimensional affine variety $Z \subset \mathbb{A}_{\mathbbm{k}}^{3g+1}$ whose points are in bijection with divisors of the form 
\[ 
\left\{ D = \sum_{i = 1}^{g} P_{i} : \hspace{0.5em} P_{i} \neq \infty \textrm{ for all } i, \hspace{0.5em} P_{i} \neq \iota(P_{j}) \textrm{ for } i \neq j \right\} .
\]
If $\mathbbm{k} = \mathbb{C}$ then the variety is smooth.
\end{thm}

\vspace{0.5em}

\noindent The equations $(\ref{jaceqs})$ therefore parametrize the points of $\Jac(C) \setminus \Theta$, where 
\[ \Theta := \left \{ [D] \in \Jac(C) : D \sim \sum_{i = 1}^{g-1} P_{i} - (g-1) \infty, \hspace{1em} P_{i} \in C(\mathbbm{k}) \textrm{ for } 1 \leq i \leq g-1  \right \} . \]

\noindent Mumford then shows that one can cover $\Jac(C)$ by an atlas of charts isomorphic to $Z$. He does this by studying sets of the form $[e_{T}] + (\Jac(C) \setminus \Theta)$, where $e_{T}$ is a $2$-torsion divisor associated to a certain subset $T$ of the branch points of $C$ (those points $P$ satisfying $\iota (P) = P$), and showing that they cover $\Jac(C)$. He then shows that the translation map $[D] \mapsto [e_{T}] + [D]$ is algebraic, and that gluing a translate of $Z$ for each set $[e_{T}] + (\Jac(C) \setminus \Theta)$ gives an atlas of charts for $\Jac(C)$. 

To describe an explicit group law on $\Jac(C)$, therefore, it suffices to describe it on $Z$. This is first and foremost because $Z$ defines a dense open set of $\Jac(C)$, and so knowing the group law on $Z$ allows one to compute it for almost all points of $\Jac(C)$ (i.e., apart from on a set of measure zero when $\mathbbm{k} = \mathbb{C}$), and secondly because if one wants to add points belonging to $\Theta$, one can pre- and post-compose with algebraic translations by $[e_{T}]$ and $-[e_{T}]$ to bring both summands into a chart isomorphic to $Z$. In principle, one has to deal with numerous edge cases corresponding to the various situations in which the translation and group-law maps may not be defined, which can occur for instance when a group addition or translation has its result in a different chart. The number and complexity of such edge cases appears to grow with $g$, and the author is unaware of an easy way to resolve them in general. For this reason, we will restrict our attention to describing the group law for divisor classes belonging to a dense open subset of $Z$, and leave a discussion of these special cases to future work.

\section{Special Classes of Polynomials}

The derivation of the group law equations will involve two operations of interest: reduction of one polynomial by another polynomial, and equating coefficients of various polynomial expressions. The process of solving equations arising from these operations has a few general features, which we develop here for use in the next section. In this section we work mainly with formal power series for simplicity, although we emphasize that in the applications that follow we will deal exclusively with polynomials. If $L$ is a Laurent series, then $[L]_{i}$ denotes its $i$'th coefficient.

For each $n \geq 1$, denote by
\[
S_{n} := \left \{ \sigma = (\sigma_{1}, \hdots, \sigma_{k}) \in \mathbb{Z}_{\geq 1}^{k} : \hspace{1em} \sum_{i = 1}^{k} \sigma_{i} = n , \hspace{0.75em} k \in \mathbb{Z}_{\geq 0} \right \}
\]
the set of compositions of the integer $n$. When $n = 0$ we adopt the usual convention that $S_{0}$ contains a single empty composition. If $\sigma \in S_{n}$ we denote by $|\sigma|$ the length of $\sigma$, which is the number of elements in the corresponding sum, or zero if $n = 0$. We have the following Lemma:

\begin{lem}
\label{reduclemma}
\emph{Suppose $\alpha = \sum_{i} \alpha_{i} x^{i}$ and $\beta = \sum_{i} \beta_{i} x^{i}$ are Laurent series over $\mathbbm{k}$. Define the \emph{$n$th iterate of the $d$th order reduction of $\alpha$ by $\beta$ at index $k$} to be the Laurent series $A_{n}$ defined inductively as follows:}
\begin{align*}
A_{-1} &= \alpha \\
A_{n} &= A_{n-1} - x^{d-n} \coeff{A_{n-1}}{k-n} \beta
\end{align*}
Then
\[ [A_{n}]_{i} = \alpha_{i} - \sum_{0 \leq \ell \leq m \leq n} \alpha_{k - \ell} \beta_{i - d + m} \left( \sum_{\sigma \in S_{m-\ell}} (-1)^{|\sigma|} \prod_{r = 1}^{|\sigma|} \beta_{k-d-\sigma_{r}} \right) .
\]
\end{lem}

\begin{rem}
The special case of Lemma \ref{reduclemma} which will be of interest is when $d \geq 0$, $\alpha$ is a polynomial of degree $k$, and $\beta$ is a monic polynomial of degree $k-d$, in which case $A_{d}$ will be the polynomial obtained by reducing $\alpha$ modulo $\beta$.
\end{rem}

\begin{proof}
For the case $n = 0$, we have 
\[
[A_{0}]_{i} = [\alpha - x^{d} \alpha_{k} \beta]_{i} = \alpha_{i} - \sum_{0 \leq \ell \leq m \leq 0} \alpha_{k - \ell} \beta_{i - d + m} \cdot 1,
\]
where the factor of $1$ can be viewed as coming from the empty product $\prod_{r = 1}^{|\sigma|} \beta_{k - d - \sigma_{r}}$ where $\sigma$ is the unique element of $S_{0}$. For the inductive case, we first observe that when $\ell \leq m \leq n-1$, the elements of $S_{n - \ell}$ are in bijection with the elements of $\bigcup_{m = \ell}^{n-1} S_{m-\ell}$, where the bijection is obtained in the natural way by adding in the last summand of $(n-m)$. We thus compute that
\begin{align*}
[A_{n}]_{i} &= [A_{n-1}]_{i} - [x^{d-n} [A_{n-1}]_{k-n} \beta]_{i} \\
&= \left[ \alpha_{i} - \sum_{0 \leq \ell \leq m \leq n-1} \alpha_{k - \ell} \beta_{i - d + m} \left( \sum_{\sigma \in S_{m - \ell}} (-1)^{|\sigma|} \prod_{r = 1}^{|\sigma|} \beta_{k - d - \sigma_{r}} \right) \right] \\
& \hspace{2em} - \left[ \alpha_{k-n} - \sum_{0 \leq \ell \leq m \leq n-1} \alpha_{k - \ell} \beta_{k - d - (n - m)} \left( \sum_{\sigma \in S_{m - \ell}} (-1)^{|\sigma|} \prod_{r = 1}^{|\sigma|} \beta_{k - d - \sigma_{r}} \right) \right] \beta_{i-d+n} \\
&= \left[ \alpha_{i} - \sum_{0 \leq \ell \leq m \leq n-1} \alpha_{k - \ell} \beta_{i - d + m} \left( \sum_{\sigma \in S_{m - \ell}} (-1)^{|\sigma|} \prod_{r = 1}^{|\sigma|} \beta_{k - d - \sigma_{r}} \right) \right] \\
& \hspace{2em} - \left[ \alpha_{k-n} \beta_{i-d+n} + \sum_{0 \leq \ell \leq n-1} \alpha_{k - \ell} \beta_{i-d+n} \left( \sum_{\ell \leq m \leq n-1} \sum_{\sigma \in S_{m - \ell}} (-1)^{|\sigma|+1} \beta_{k - d - (n - m)} \prod_{r = 1}^{|\sigma|} \beta_{k - d - \sigma_{r}} \right) \right] \\
&= \left[ \alpha_{i} - \sum_{0 \leq \ell \leq m \leq n-1} \alpha_{k - \ell} \beta_{i - d + m} \left( \sum_{\sigma \in S_{m - \ell}} (-1)^{|\sigma|} \prod_{r = 1}^{|\sigma|} \beta_{k - d - \sigma_{r}} \right) \right] \\
& \hspace{2em} - \left[ \alpha_{k-n} \beta_{i-d+n} + \sum_{0 \leq \ell \leq n-1} \alpha_{k - \ell} \beta_{i-d+n} \left( \sum_{\sigma \in S_{n - \ell}} (-1)^{|\sigma|} \prod_{r = 1}^{|\sigma|} \beta_{k - d - \sigma_{r}} \right) \right] \\
&= \alpha_{i} - \sum_{0 \leq \ell \leq m \leq n} \alpha_{k - \ell} \beta_{i - d + m} \left( \sum_{\sigma \in S_{m-\ell}} (-1)^{|\sigma|} \prod_{r = 1}^{|\sigma|} \beta_{k-d-\sigma_{r}} \right) .
\end{align*}
\end{proof}

The next special situation of interest arises when equating coefficients of two polynomials, one of which arises from a product. We again work in the language of formal power series for convenience.

\begin{lem}
\label{usolverlemma}
Suppose that $\alpha = \sum_{i \geq 0} \alpha_{i} x^{i}$, $\beta = \sum_{i \geq 0} \beta_{i} x^{i}$ and $\gamma = \sum_{i \geq 0} \gamma_{i} x^{i}$ are formal power series over $\mathbbm{k}$, and that $\alpha = \beta \gamma$. Suppose also that $\gamma_{0} \neq 0$. Then we are in the situation that $\alpha_{k} = \sum_{j = 0}^{k} \beta_{j} \gamma_{k-j}$, and so 
\begin{align*}
\beta_{k} = \sum_{j = 0}^{k} \frac{\alpha_{j}}{\gamma_{0}} \sum_{\sigma \in S_{k-j}} \frac{(-1)^{|\sigma|}}{\gamma_{0}^{|\sigma|}} \prod_{r = 1}^{|\sigma|} \gamma_{\sigma_{r}} .
\end{align*}
\end{lem}

\begin{proof}
For $k = 0$ we have $\alpha_{0} = \beta_{0} \gamma_{0}$, and so we may invert $\gamma_{0}$ to get the desired equation. Considering the inductive case, we have that $\alpha_{k} = \sum_{i = 0}^{k-1} \beta_{i} \gamma_{k-i} + \beta_{k} \gamma_{0}$, and so
\begin{align*}
\beta_{k} &= \frac{\alpha_{k}}{\gamma_{0}} + \sum_{i = 0}^{k-1} \frac{(-1)}{\gamma_{0}} \beta_{i} \gamma_{k-i} \\
&= \frac{\alpha_{k}}{\gamma_{0}} + \sum_{i = 0}^{k-1} \frac{(-1)}{\gamma_{0}} \left( \sum_{j = 0}^{i} \frac{\alpha_{j}}{\gamma_{0}} \sum_{\sigma \in S_{i-j}} \frac{(-1)^{|\sigma|}}{\gamma_{0}^{|\sigma|}} \prod_{r = 1}^{|\sigma|} \gamma_{\sigma_{r}} \right) \gamma_{k-i} \\
&= \frac{\alpha_{k}}{\gamma_{0}} + \sum_{j = 0}^{k-1} \frac{\alpha_{j}}{\gamma_{0}} \sum_{j \leq i \leq k-1} \sum_{\sigma \in S_{i-j}} \frac{(-1)^{|\sigma| + 1}}{\gamma_{0}^{|\sigma| + 1}} \gamma_{k-i} \prod_{r = 1}^{|\sigma|} \gamma_{\sigma_{r}} \\
&= \sum_{j = 0}^{k} \frac{\alpha_{j}}{\gamma_{0}} \sum_{\sigma \in S_{k-j}} \frac{(-1)^{|\sigma|}}{\gamma_{0}^{|\sigma|}} \prod_{r = 1}^{|\sigma|} \gamma_{\sigma_{r}} ,
\end{align*}
where we have used the natural bijection between $S_{k-j}$ and $\bigcup_{i = j}^{k-1} S_{i-j}$ obtained by adding $(k-i)$.
\end{proof}

\section{The Group Law}

To compute the sum of two distinct points $P$ and $Q$ (representing the divisor classes $[P - \infty]$ and $[Q - \infty]$) on an elliptic curve, one intersects the curve $C$ with a line $\ell$ through $P$ and $Q$ which intersects the curve $C$ at a third point $R$. The sum $[P - \infty] + [Q - \infty]$ is then the divisor class $[\iota(R) - \infty]$, and equations for the group law may be computed by explicitly solving the curve equation for the coordinates of the point $\iota(R)$.

To generalize this strategy to a hyperelliptic curve of genus $g$, it is natural to try adding $[D_{1}] = [P_{1} + \cdots + P_{g} - g \infty]$ to $[D_{2}] = [Q_{1} + \cdots + Q_{g} - g \infty]$ by constructing an interpolating function $\ell(x)$ through the points $P_{1}, \hdots, P_{g}, Q_{1}, \hdots, Q_{g}$ which intersects the curve at $g$ other points $R_{1}, \hdots, R_{g}$. The sum $[D_{1}] + [D_{2}]$ is then the divisor class $[\iota(R_{1}) + \cdots + \iota(R_{g}) - g \infty]$. If one then attempts to solve for the coordinates of the points $\iota(R_{1}), \hdots, \iota(R_{g})$, however, this seems to require extracting roots, and so this strategy does not produce rational formulas for the group law. 

An alternative strategy, employed in the work of Costello and Lauter\cite{costlaut}, is to instead represent the divisors $D_{1}$ and $D_{2}$ using two pairs $(u_{1}, v_{1})$ and $(u_{2}, v_{2})$ as in Section \ref{jacSection}. If one does this, then the condition that the interpolation function $\ell$ intersect the curve $C$ with appropriate multiplicity at the various points $P_{i}$ and $Q_{i}$ for $1 \leq i \leq g$ becomes equivalent to the two modular conditions $v_{1} \equiv \ell \pmod{u_{1}}$ and $v_{2} \equiv \ell \pmod{u_{2}}$. Performing a modular reduction, one gets a linear system of equations for the coefficients of $\ell$, and solves them to find the interpolation function $\ell$ in terms of the coefficients of $u_{1}, v_{1}, u_{2}$ and $v_{2}$. Noting that the function $f - \ell^2$ vanishes on all the points $P_{i}, Q_{i}$ and any additional intersections $R_{j}$, one can then derive linear relations for the coefficients of a polynomial $u_{3}$ whose roots give the $x$-coordinates of the points $R_{j}$ by noting that $u_{3} | (f - \ell^2)$; it is then a simple matter to find an appropriate $v_{3}$ to describe the sum.

Costello and Lauter carry out this strategy explicitly for $g = 2$, and sketch how it might work in general, but their approach has an important drawback. Namely, the interpolation functions they use are simply polynomials in $x$, and for $g > 2$ they do not give $g$ additional intersections $R_{1}, \hdots, R_{g}$ but instead some number of intersections strictly between $g$ and $2g$. Therefore, their strategy requires carrying out multiple stages of calculations, the number of which depends on $g$, and appropriate formulas must be derived for each choice of $g$ independently. Ideally, it would be possible to carry out a similar strategy with an interpolation function for which exactly $g$ additional intersections $R_{1}, \hdots, R_{g}$ are guaranteed in the general case, and so do the computation ``all at once''.

To achieve such an interpolation of the points $P_{1}, \hdots, P_{g}, Q_{1}, \hdots, Q_{g}$, we use rational functions of the form 
\begin{equation}
\label{interp}
\frac{p(x)}{q(x)} = \frac{p_{a} x^{a} + \cdots + p_{1} x + p_{0}}{q_{b} x^{b} + \cdots + q_{1} x + q_{0}} ,
\end{equation}
where $a = (3g - \ep)/2$, $b = (g-2+\ep)/2$, and $\ep$ is the parity of $g$. Since we have $\deg p + \deg q + 2 = 2g + 1$ coefficients and only $2g$ points to interpolate, we have one additional degree of freedom. The interpolation function is a polynomial of degree $1$ (respectively $3$) for the cases $g = 1$  (respectively $g = 2$). Such interpolation functions are considered by Leitenberger in his paper\cite{lieten}, and were first considered by Jacobi\cite{jacobi} in connection with Abel's Theorem. Leitenberger uses these interpolation functions to derive equations for the group law in the $g = 2$ case, but his methods do not appear to generalize. Our derivation, which will be more in line with the polynomial division techniques used in the paper \cite{costlaut} of Costello and Lauter, will achieve explicit formulas for all positive integers $g$.

\subsection{Group Law on a Dense Open Set}

Recall that, by the discussion in Section \ref{jacSection}, we are working to describe the group law on the open dense set $Z$ described in Theorem \ref{mumfDenseOpen}. The points of $Z$ are in bijection with unordered tuples of $g$ points on $C_{1}$, none of which are conjugates of each other. The variety $Z$ is described by $2g+1$ equations in the coefficients of three polynomials $u, v$ and $w$, however the coefficients of $w$ are entirely determined by those of $u$ and $v$ so we may ignore $w$ and simply use the polynomials $u$ and $v$. 

The derivation takes the form of a series of three lemmas. The first of these, Lemma \ref{pqlem}, derives equations for the interpolation function $p/q$ in terms of the coefficients of two pairs $(u, v)$ and $(u', v')$ representing divisors $D$ and $D'$. The second lemma, Lemma \ref{udplemma}, uses the relationship between $p/q$ and $f$ to find formulas for the coefficients of a degree $g$ monic polynomial $u''$ representing the $x$-coordinates of a divisor $D''$ which corresponds to the sum $[D] + [D']$. The third and final lemma solves for the coefficients of the polynomial $v''$ in terms of the coefficients of $u''$ and $p/q$.

\newpage

\begin{lem}
\label{pqlem}
Suppose that $(u, v)$ and $(u', v')$ describe divisors $D = \sum_{i = 1}^{g} P_{i}$ and $D' = \sum_{i = 1}^{g} P_{i}'$ respectively, such that the summands in $D$ have $x$-coordinates which are distinct from the $x$-coordinates of the summands in $D'$. Let $a = (3g - \ep)/2$ and $b = (g-2+\ep)/2$ as before, and define $d = (g - \ep)/2$. Define the quantities: 
\begin{align*}
\kappa_{i, \ell} &= \sum_{\ell \leq m \leq d} u_{i-d+m} \left( \sum_{\sigma \in S_{m - \ell}} (-1)^{|\sigma|} \prod_{r = 1}^{|\sigma|} u_{g-\sigma_{r}} \right) \\
\kappa'_{i, \ell} &= \sum_{\ell \leq m \leq d} u'_{i-d+m} \left( \sum_{\sigma \in S_{m - \ell}} (-1)^{|\sigma|} \prod_{r = 1}^{|\sigma|} u'_{g-\sigma_{r}} \right) \\
\lambda_{i, j} &= - v_{i-j} + \sum_{0 \leq \ell \leq d} v_{(a-j)-\ell} \kappa_{i, \ell} \\
\lambda_{i, j}' &= - v_{i-j}' + \sum_{0 \leq \ell \leq d} v_{(a-j)-\ell}' \kappa_{i, \ell}'
\end{align*}
Then the requirement that a rational function of the form in (\ref{interp}) interpolates the divisors $D$ and $D'$ induces the following system of linear relations on the coefficients of $p/q$:

\begin{centerlongequation}
\left( \begin{array}{ccc|ccc}
\kappa_{0,d} - \kappa'_{0,d} & \cdots & \kappa_{0,0} - \kappa'_{0,0} & \lambda'_{0,1} - \lambda_{0,1} & \cdots & \lambda'_{0,b} - \lambda_{0,b} \\
\kappa_{1,d} - \kappa'_{1,d} & \cdots & \kappa_{1,0} - \kappa'_{1,0} & \lambda'_{1,1} - \lambda_{1,1} & \cdots & \lambda'_{1,b} - \lambda_{1,b} \\
\kappa_{2,d} - \kappa'_{2,d} & \cdots & \kappa_{2,0} - \kappa'_{2,0} & \lambda'_{2,1} - \lambda_{2,1} & \cdots & \lambda'_{2,b} - \lambda_{2,b} \\
\vdots & \ddots & \vdots & \vdots & \ddots & \vdots \\
\kappa_{g-2,d} - \kappa'_{g-2,d} & \cdots & \kappa_{g-2,0} - \kappa'_{g-2,0} & \lambda'_{g-2,1} - \lambda_{g-2,1} & \cdots & \lambda'_{g-2,b} - \lambda_{g-2,b} \\
\kappa_{g-1,d} - \kappa'_{g-1,d} & \cdots & \kappa_{g-1,0} - \kappa'_{g-1,0} & \lambda'_{g-1,1} - \lambda_{g-1,1} & \cdots & \lambda'_{g-1,b} - \lambda_{g-1,b} \\
\end{array} \right)
\begin{pmatrix} p_{g}/q_{0} \\ \vdots \\ p_{a}/q_{0} \\ \midrule q_{1}/q_{0} \\ \vdots \\ q_{b}/q_{0} \end{pmatrix}
=
\begin{pmatrix} \lambda_{0,0}-\lambda'_{0, 0} \\ \lambda_{1,0}-\lambda'_{1, 0} \\ \lambda_{2,0}-\lambda'_{2, 0} \\ \vdots \\ \lambda_{g-2,0}-\lambda'_{g-2, 0} \\ \lambda_{g-1,0}-\lambda'_{g-1, 0} \end{pmatrix}
\end{centerlongequation}
\[
p_{i} = \sum_{\ell = 0}^{d} p_{a - \ell} \kappa_{i, \ell} - \sum_{j = 1}^{b} q_{j} \lambda_{i, j} - q_{0} \lambda_{i, 0} \hspace{7em} 0 \leq i \leq g-1
\]
Label the $g \times g$ matrix $M$, and let $M_{j}$ denote the matrix obtained from $M$ by replacing the $j$th column with the solution vector on the right. Then on a dense open set of $Z \times Z$ these relations determine an interpolation function $p/q$ with the desired properties via the equations
\begin{align*}
p_{g+i} &= \det(M_{1+i})		& 0 \leq i \leq d+1 \\
q_{0} &= \det(M)				& \\
q_{i} &= \det(M_{1+d+i})		& 1 \leq i \leq b \\
p_{i} &= \sum_{\ell = 0}^{d} \det(M_{1+(d-\ell)}) \kappa_{i, \ell} - \sum_{j = 1}^{b} \det(M_{1+d+j}) \lambda_{i, j} - \det(M) \lambda_{i, 0} & 0 \leq i \leq g-1 
\end{align*}

\end{lem}

\begin{proof}
Label the points $P_{i} = (x_{i}, y_{i})$ and $P_{i}' = (x_{i}', y_{i}')$. The requirement that $p/q$ interpolates the points of $D$ is equivalent to the condition that $p/q \equiv v \pmod{u}$, and since we require that $q$ does not vanish at any $x_{i}$, to the condition that $p - qv \equiv 0 \pmod{u}$. By expanding this relation, we see that this condition is equivalent to
\begin{align*}
\sum_{i \geq 0} \alpha_{i} x^{i} := \sum_{i \geq 0} \left( p_{i} - \sum_{j = 0}^{b} q_{j} v_{i-j} \right) x^{i} \equiv 0 \pmod{u} .
\end{align*}
To find appropriate linear relations for the coefficients of $p/q$, we apply Lemma \ref{reduclemma} with $\beta_{i} = u_{i}$, $n = d = a-g = (g - \ep)/2$ and $k = a$. We therefore get for $0 \leq i \leq g-1$ the relations
\begin{align*}
0 &= \alpha_{i} - \sum_{\ell = 0}^{d} \alpha_{a - \ell} \kappa_{i, \ell} \\
&= \left( p_{i} - \sum_{j = 0}^{b} q_{j} v_{i-j} \right) - \sum_{\ell = 0}^{d} \left( p_{a - \ell} - \sum_{j = 0}^{b} q_{j} v_{a - \ell - j} \right) \kappa_{i, \ell} \\
&= p_{i} - \sum_{\ell = 0}^{d} p_{a - \ell} \kappa_{i, \ell} +  \sum_{j = 1}^{b} q_{j} \left( - v_{i-j} + \sum_{\ell = 0}^{d} v_{(a-j) - \ell} \kappa_{i, \ell} \right) + q_{0} \left( - v_{i} + \sum_{\ell = 0}^{d} v_{a - \ell} \kappa_{i, \ell} \right)
\end{align*}
Using the notation defined in the statement of the Lemma, this reads
\begin{align}
\label{pqlemfstrel}
0 &= p_{i} - \sum_{\ell = 0}^{d} p_{a - \ell} \kappa_{i, \ell} +  \sum_{j = 1}^{b} q_{j} \lambda_{i, j} + q_{0} \lambda_{i, 0}   & 0 \leq i \leq g-1.
\end{align}
The analogous process for the primed variables gives us the same equations with $\kappa'_{i,\ell}$ replacing $\kappa_{i,\ell}$ and $\lambda'_{i,j}$ replacing $\lambda_{i,j}$. Therefore, taking differences we see that in order for $p/q$ to have the desired form, we must have 
\begin{align}
\label{pqlemsndrel}
0 &= \sum_{\ell = 0}^{d} p_{a - \ell} (\kappa_{i, \ell} - \kappa'_{i, \ell}) +  \sum_{j = 1}^{b} q_{j} (\lambda'_{i,j} - \lambda_{i, j}) + q_{0} (\lambda'_{i, 0} - \lambda_{i, 0})  & 0 \leq i \leq g-1.
\end{align}
Equation (\ref{pqlemsndrel}) gives the matrix equation after dividing through by $q_{0}$, and equation (\ref{pqlemfstrel}) gives the desired relation for $p_{i}$ for $0 \leq i \leq g-1$. The formulas for the coefficients of $p/q$ then follow by Cramer's rule, assuming that the linear system is non-degenerate.

We now show that the matrix $M$ is non-degenerate on a dense open set of $Z \times Z$. Note that because $Z \times Z$ is irreducible and $\det(M) \neq 0$ is an open condition, it suffices to show that the set of points for which $M$ is non-degenerate is non-empty. Note that the conditions $p/q \equiv v \pmod{u}$ and $p/q \equiv v' \pmod{u'}$ uniquely determine $p/q$ up to a projective rescaling, since if $\widetilde{p}/\widetilde{q}$ is another interpolation function satisfying the same conditions we have $p \widetilde{q} \equiv \widetilde{p} q \pmod{u u'}$ and hence $p \widetilde{q} = \widetilde{p} q$ since $\deg (p \widetilde{q} - \widetilde{p} q) = 2g-1 < 2g = \deg(u u')$. Since the derived linear system is equivalent to the condition that $p/q$ is an interpolation function of the desired form, the statement that the system is solvable on an open dense set of $Z \times Z$ amounts to the statement that at least one such interpolation function exists, which is clearly true.

\end{proof}

\begin{lem}
\label{udplemma}
Continue with the notation and assumptions of Lemma \ref{pqlem}. Define the quantities: 
\begin{align*}
\rho &= p_{a}^2 ( 1 - \ep) - f_{2g+1} q_{b}^2 \ep \\
\omega_{j} &= \sum_{i = 0}^{j} u_{i} u'_{j-i} \\
\eta_{k} &= \sum_{j = 0}^{k} \left( p_{j} p_{k-j} - f_{k-j} \sum_{i = 0}^{j} q_{i} q_{j-i} \right)
\end{align*}
Suppose the sum $[D - g \infty] + [D' - g \infty]$ is represented by a divisor $D'' - g \infty$ with $D'' = \sum_{i = 1}^{g} P_{i}''$ and $P_{i}''$ a point on $C_{1}$ for $1 \leq i \leq g$. Then if $(u'', v'')$ is the pair of polynomials representing $D''$, the coordinates of $u''$ are given by:
\[
u''_{j} = \sum_{i = 0}^{j} \frac{\eta_{i}}{\rho \omega_{0}} \sum_{\sigma \in S_{j-i}} \frac{(-1)^{|\sigma|}}{\omega_{0}^{|\sigma|}} \prod_{r = 1}^{|\sigma|} \omega_{\sigma_{r}} .
\]
\end{lem}
\begin{proof}
The polynomials $p$ and $q$ in Lemma \ref{pqlem} were computed to satisfy $p - qv \equiv 0 \pmod{u}$. Furthermore, the pair $(u, v)$ satisfies $f - v^2 \equiv 0 \pmod{u}$. Together these two facts imply that
\[ p^2 - f q^2 \equiv p^2 - v^2 q^2 \equiv (p - q v)^2 + 2qv(p - qv) \equiv 0 +  0 \equiv 0 \pmod{u} . \]
The analogous fact is true for $u'$. Since $u$ and $u'$ do not share roots, we see that $uu'|(p^2 - fq^2)$. The polynomial $p^2 - f q^2$ has degree $\max \{ 2 a, 2(b+g)+1 \} = 3g$ with leading coefficient $\rho$, and so we may write $p^2 - fq^2 = \rho u u' u''$ where $u''$ is monic of degree $g$ and the roots $x''_{i}$ of $u''$ are such that there exists $Q_{i} = (x''_{i}, y''_{i})$ on $C_{1}$ satisfying $p(x''_{i}) - y''_{i} q(x''_{i}) = 0$. 

Viewing $p - y q$ as a function on $C$, it has zeros precisely at the roots of the polynomial $p^2 - f q^2$, and so has $3g$ of them (with multiplicity) corresponding to the roots of the polynomials $u, u'$ and $u''$. As the number of zeros on $C$ must equal the number of poles, the function $p - y q$ must then have a pole of order $3g$ at $\infty$, and so we find that 
\[ (D - g\infty) + (D' - g \infty) \sim - \sum_{i = 1}^{g} Q_{i} + g \infty . \]
The relations $Q_{i} + \iota(Q_{i}) \sim 2 \infty $ then give us that 
\[ (D - g \infty) + (D' - g \infty) \sim \sum_{i = 1}^{g} \iota(Q_{i}) - g \infty . \]
So we see that if we take $P_{i}'' = (x''_{i}, - y''_{i})$, then $u''$ satisfies the hypotheses of the theorem.

To solve for the coefficients $u''_{j}$, we expand the relation $ p^2 - f q^2 = \rho u u' u''$ and equate coefficients. This gives us:
\[ \sum_{j = 0}^{k} \left( p_{j} p_{k-j} -  f_{k-j} \sum_{i = 0}^{j} q_{i} q_{j-i} \right) = \rho \sum_{j = 0}^{k} u''_{j} \left( \sum_{i = 0}^{k-j} u_{i} u'_{(k-j)-i} \right) , \]
or simply $\eta_{k}/\rho = \sum_{j = 0}^{k} u''_{j} \omega_{k-j}$. Applying Lemma \ref{usolverlemma} gives the result.
\end{proof}

\begin{rem}
The formulas in Lemma \ref{udplemma} are defined provided that $\omega_{0} \neq 0$ and $\rho \neq 0$. The first condition reduces to the statement that $x_{i} \neq 0$ and $x'_{i} \neq 0$ for all $1 \leq i \leq g$, and the second says that either $p_{a} \neq 0$ or $q_{b} \neq 0$ depending on the parity of $g$. This latter case again reduces to the non-vanishing of a certain matrix determinant as defined in Lemma \ref{pqlem}, which again defines a dense open subset of $Z \times Z$ for similar reasons as before. 
\end{rem}

\begin{lem}
\label{vdplemma}
Continue with the notation and assumptions in Lemmas \ref{pqlem} and \ref{udplemma}. Define the quantities:
\begin{align*}
\kappa''_{i, \ell} &= \sum_{\ell \leq m \leq d} u''_{i-d+m} \left( \sum_{\sigma \in S_{m - \ell}} (-1)^{|\sigma|} \prod_{r = 1}^{|\sigma|} u''_{g-\sigma_{r}} \right) \\
\tau_{i, s} &= - \sum_{m = g+1-\ep}^{d+s} q_{a-m} \kappa''_{i, m-s} \\
\mu_{i} &= -p_{i} + \sum_{0 \leq \ell \leq d} p_{a-\ell} \kappa''_{i, \ell}
\end{align*}
Then we have

\begin{centerlongequation}
\left[\begin{pmatrix}
q_{0} & 0 & \cdots & 0 & 0 & \cdots & 0 \\
q_{1} & q_{0} & \cdots & 0 & 0 & \cdots & 0 \\
\vdots & \vdots & \ddots & \vdots & \vdots & \ddots & \vdots \\
q_{b} & q_{b-1} & \cdots & q_{0} & 0 & \cdots & 0 \\ 
0 & q_{b} & \cdots & q_{1} & q_{0}  & \cdots  & 0 \\
\vdots & \vdots & \ddots & \ddots  & \ddots  & \ddots  & \vdots \\
0 & 0 & \cdots & q_{b} & q_{b-1} & \cdots & q_{0} \\
\end{pmatrix} 
+
\begin{pmatrix}
0 & \cdots & 0 & \tau_{0, d+1} & \cdots & \tau_{0, g-1} \\
0 & \cdots & 0 & \tau_{1, d+1} & \cdots & \tau_{1, g-1} \\
0 & \cdots & 0 & \tau_{2, d+1} & \cdots & \tau_{2, g-1} \\
\vdots & \vdots & \vdots & \vdots & \vdots & \vdots \\
\vdots & \vdots & \vdots & \vdots & \vdots & \vdots \\
0 & \cdots & 0 & \tau_{g-2, d+1} & \cdots & \tau_{g-2, g-1} \\
0 & \cdots & 0 & \tau_{g-1, d+1} & \cdots & \tau_{g-1, g-1} \\
\end{pmatrix}\right]
\begin{pmatrix}
v''_{0} \\ v''_{1} \\ v''_{2} \\ \vdots \\ \vdots \\ v''_{g-2} \\ v''_{g-1}
\end{pmatrix}
=
\begin{pmatrix}
\mu_{0} \\ \mu_{1} \\ \mu_{2} \\ \vdots \\ \vdots \\ \mu_{g-2} \\ \mu_{g-1}
\end{pmatrix}
\end{centerlongequation}
\end{lem}
\noindent and so
\[ v''_{i} = \frac{\det (Q + T)_{1+i}}{\det (Q+T)}, \]
where $Q + T$ is the sum of the two matrices between the square brackets, with $Q$ denoting the first matrix and $T$ the second, and $(Q+T)_{j}$ is the matrix obtained by replacing the $j$th column of $Q+T$ with the solution vector on the right.

\begin{proof}
As with the pairs $(u, v)$ and $(u', v')$ we have a relation $p + qv'' \equiv 0 \pmod{u''}$, this time with a sign change to account for the sign of the $y$-coordinate in the points $P''_{i}$. Proceeding as in Lemma \ref{pqlem}, we have the equations
\begin{align*}
0 &= p_{i} - \sum_{\ell = 0}^{d} p_{a - \ell} \kappa''_{i, \ell} +  \sum_{j = 0}^{b} q_{j} \left(v''_{i-j} - \sum_{\ell = 0}^{d} v''_{(a-j) - \ell} \kappa''_{i, \ell} \right) \\
&= - \mu_{i} + \left(\sum_{j = 0}^{g-1} v''_{j} q_{i-j}\right) - \left( \sum_{j = 0}^{b} \sum_{\ell = 0}^{d} v''_{(a-j)-\ell} q_{j}  \kappa''_{i, \ell} \right)
\end{align*}
To extract the coefficient of $v''_{(a-j)-\ell}$ in the second summation on the last line, we use the change of indices $s = (a-j)-\ell$ and $m = a-j = s+\ell$. As $s$ is the index of $v''$ we have the bound $s \leq g-1$, and from the equality $s = (a-j)-\ell$ we get $s \geq a-b-d = d+1$. Then for fixed $s$, we have $m \leq s + d$ and $m \geq a - b = g+1-\ep$. This gives us the equality
\[ \sum_{j = 0}^{g-1} v''_{j} q_{i-j} + \sum_{s = d+1}^{g-1} v''_{s} \left( - \sum_{m = g+1-\ep}^{s+d} q_{a-m}  \kappa''_{i, m-s} \right) = \mu_{i} , \]
from which the matrix equation follows. The formula for $v''_{i}$ then follows from Cramer's rule.
\end{proof}

\begin{rem}
To understand when the above formulas successfully determine $v''$ (in particular, when $\det(Q + T)$ does not vanish), note that if the roots of $u''$ are distinct and do not coincide with the roots of $q$, then the relationship $p + q v'' \equiv 0 \pmod{u''}$ determines the value of the degree $g-1$ polynomial $v''$ at $g$ distinct points, which suffices to determine it. These conditions on $u''$ and $q$ may be expressed by asserting the non-vanishing of certain discriminant and resultant polynomials, so we once again see that the desired relations hold on some dense open set of $Z \times Z$.
\end{rem}

\begin{thm}
There exist explicit polynomial and rational functions describing the group law on an open dense set of $\Jac(C)$.
\end{thm}

\begin{proof}This is merely a summary of Lemmas \ref{pqlem}, \ref{udplemma} and \ref{vdplemma} and their associated remarks.
\end{proof}

\section{Conclusion}

The formulas we have described have some drawbacks compared to the usual methods for computing the group law. For one, the use of inversions and the requirements on both $Z$ and the divisors represented by $(u, v)$ and $(u', v')$ limit the scope of the formulas somewhat, and one might suspect that handling the various edge cases would make them difficult to use. In fact, this is generally not so serious, since most applications of abelian variety arithmetic in cryptography or computer science require the use of finite fields of exponentially large prime characteristic $r$, and if one heuristically models each inequality defining the validity of the group law as holding with probability $(r-1)/r$, then one concludes that encountering most such edge cases is exponentially unlikely in practice.

Another objection is that the formulas do not extend to the important case of doubling. This is already the case when $g=1$, which as shown in Appendix \ref{ellapp} is really just the case of the usual elliptic curve group law, where the chord-based addition formula only holds when adding two distinct points and one must instead use a tangent line in the degenerate case. A similar phenomenon holds here, in that when doubling points the relations in Lemma \ref{pqlem} are always dependent, and one must use additional relations which enforce a higher-order agreement between the interpolation function $p/q$ and the function $y$ on $C$ to determine $p/q$. This is done for the case $g = 2$ in the work of Costello and Lauter, but we do not pursue it here as the approach grows considerably in complexity with $g$. However we may simply observe that one can circumvent this issue entirely by simply computing a scaling of the form $2 [D]$ as a sum of the form $(([D] + [E]) + [D]) - [E]$, where $[E]$ is an appropriate ``dummy'' divisor class chosen at random.

Another objection is that the formulas use expressions that grow quickly in complexity, requiring sums over compositions and matrix determinants, and so are unlikely to be competitive with reduction-based approaches for large $g$. While this is certain to be true asymptotically, the $g = 1$ and $g = 2$ cases (that of the elliptic curve group law and the work of Costello and Lauter respectively) are quite efficient, and a heavily unoptimized implementation by the author\cite{jaccode} was able to use the formulas for Jacobian arithmetic up to $g = 8$ without much difficulty. We note that the general expressions that appear in Lemmas \ref{pqlem} and \ref{vdplemma} obscure the fact that many of the terms that appear in these expressions are often zero (either due to an abundance of zeros in the coefficients of $f$ or because the indices fall out of range), and so in practice the complexity may be overstated. The case where $g=3$ in particular may benefit from some hand-optimization.

We also wish to emphasize the inherent value in explicit constructions. The usual approach to constructing the Jacobian of a curve as an abelian variety uses the language of schemes and representable functors, which is convenient for many theoretical purposes, but carries with it associated baggage that can make it difficult to apply. For this reason, the use of higher-dimensional abelian varieties in cryptography and computer science can often be traced back to either the hyperelliptic Jacobian construction appearing in Mumford's Lectures on Theta, or the work of Flynn, even though it is unlikely those authors had any particular computational application in mind. These constructions are messy, but they can be made practical, whereas the author is unaware of any computational applications of the usual scheme-theoretic approach.

\section{Acknowledgements}

The author thanks Matt Satriano and Jerry Wang for helpful comments on a draft of this manuscript.

\bibliography{arxiv_aug_2018}{}
\bibliographystyle{plain}

\appendix 

\section{The Elliptic Curve Case}
\label{ellapp}

As an illustrative example, we demonstrate that the above derivation gives the usual group law in the case $g = 1$. The equation defining $C_{1}$ is
\[ y^2 = f(x) = x^{3} + f_{2} x^2 + f_{1} x + f_{0} . \]
A pair representing the divisor $D = P$ looks like $(u, v) = (x + u_{0}, v_{0})$. The open set $Z$ is then described by the equations (\ref{jaceqs}), which are
\begin{align*}
f_{0} - v_{0}^2 &= u_{0}w_{0} \\
f_{1} &= u_{0} w_{1} + w_{0} \\
f_{2} &= u_{0} + w_{1} .
\end{align*}
Using the second equation we may eliminate $w_{0}$, and using the third equation we may further eliminate $w_{1}$, resulting in a curve defined by
\begin{align*}
f_{0} - v_{0}^2 &= u_{0}(f_{1} - u_{0}(f_{2} - u_{0})) \\
f_{0} - v_{0}^2 &= u_{0}f_{1} - u_{0}^2 f_{2} + u_{0}^3 \\
v_{0}^2 &= (- u_{0})^3 + f_{2} (- u_{0})^2 + f_{1} (- u_{0}) + f_{0} ,
\end{align*}
which is evidently isomorphic to $C_{1}$. 

Now let $D = P$, $D' = P'$, $(u, v) = (x + u_{0}, v_{0})$ and $(u', v') = (x + u_{0}', v_{0}')$. Following the notation in Lemma \ref{pqlemfstrel} we have $\ep = 1$ and so $a = 1$ and $b = 0$. Hence the interpolation function $p/q$ is of the form $(p_{1}/q_{0}) x + (p_{0} / q_{0})$. The matrix in Lemma \ref{pqlemfstrel} is $1 \times 1$ with a single entry 
\[ \kappa_{0, 0} - \kappa'_{0, 0} = u_{0} - u_{0}' , \]
and the solution vector is also $1 \times 1$ with a single entry 
\[ \lambda_{0, 0} - \lambda_{0, 0}' = v_{0}' - v_{0} . \]
We therefore get, from the formulas in Lemma \ref{pqlem}, that 
\begin{align*}
p_{1} &= v_{0}' - v_{0} \\
q_{0} &= u_{0} - u_{0}' \\
p_{0} &= (v_{0}' - v_{0}) u_{0} - (u_{0} - u_{0}') (-v_{0}) ,
\end{align*}
and hence
\[ \frac{p(x)}{q(x)} = \frac{v_{0}' - v_{0}}{u_{0} - u_{0}'} x + \frac{v_{0}' - v_{0}}{u_{0} - u_{0}'} u_{0} + v_{0} . \]
One easily checks that $p/q$ is a line through the points $(-u_{0}, v_{0})$ and $(-u_{0}', v_{0}')$. 

Continuing with Lemma \ref{udplemma}, we see that 
\begin{align*}
u_{0}'' &= \frac{p_{0}^2 - f_{0} q_{0}^2}{- q_{0}^2 u_{0} u_{0}' }
\end{align*}
A long but straightforward calculation shows that $u_{0}'' = -\lambda^2 + f_{2} - u_{0} - u_{1}$, where $\lambda = (v_{0}' - v_{0})/(u_{0} - u_{0}')$. This agrees with the usual formulas for the elliptic curve group law for a Weirstrass form elliptic curve. Note that $u_{0}''$ is the negative of the usual $x$-coordinate here. Then, applying Lemma \ref{vdplemma} we get that 
\begin{align*}
v_{0}'' &= \frac{\mu_{0}}{q_{0}} \\ 
&= \frac{-p_{0} + p_{1} \kappa''_{0,0}}{q_{0}} \\
&= - \lambda u_{0} + \lambda u''_{0} - v_{0} ,
\end{align*}
which also agrees with the usual formulas.
\end{document}